\documentclass[11pt, reqno, a4paper%, draft
]{amsart}

\usepackage{amsthm,amssymb,mathrsfs}
\usepackage{microtype,xspace,calc}
\usepackage{enumerate}
\usepackage{amsfonts}
\font\rus=wncyr10

\usepackage[paper=a4paper,text={14.5cm,22cm}]{geometry}
\usepackage[colorlinks=true, citecolor=blue, urlcolor=blue, linkcolor=blue]{hyperref}

 \title[Derivation representation of mixed elliptic motive~Lie algebra]
{On the derivation representation of the fundamental 
Lie algebra of mixed elliptic motives}

\author{S. Baumard and L. Schneps}
\date{\today}

\theoremstyle{plain}
\newtheorem{theo}{Theorem}[section]
\newtheorem{lem}[theo]{Lemma}
\newtheorem{prop}[theo]{Proposition}
\newtheorem{cor}[theo]{Corollary}
\newtheorem{conj}[theo]{Conjecture}

\theoremstyle{definition}
\newtheorem{defi}[theo]{Definition}
\newtheorem*{defis}{Definitions}

\theoremstyle{remark}

\newcommand{\seg}[1]{\mathopen{[\![} #1 \mathclose{]\!]}}

\newcommand{\nlineshort}{{}\nonumber\\&{}}
\newcommand{\nline}{\nlineshort \hspace*{2em}}
\newcommand{\pol}{\textnormal{pol}}
\newcommand{\den}[2]{\De_{#1}({#2})}

\newcommand{\De}{\hbox{\rus{D}}}
\newcommand{\E}{\mathscr{E}}

\DeclareMathOperator{\sham}{\textnormal{sha}}
\DeclareMathOperator{\Lie}{Lie}

\renewcommand{\phi}{\varphi}
\newcommand{\eps}{\varepsilon}

\renewcommand{\geq}{\geqslant}
\renewcommand{\leq}{\leqslant}
\newcommand{\Q}{\mathbb{Q}}
\newcommand{\Z}{\mathbb{Z}}

\newcommand{\chev}[1]{\langle #1 \rangle}

\newcommand{\grt}{\mathfrak{grt}}

\newcommand{\pls}{\mathfrak{pls}}
\newcommand{\ls}{\mathfrak{ls}}

\DeclareMathOperator{\Der}{Der}
\renewcommand{\sl}{\mathfrak{sl}_2}

\DeclareMathOperator{\muu}{mu}
\DeclareMathOperator{\ma}{ma}
\DeclareMathOperator{\ARI}{ARI}
\DeclareMathOperator{\ari}{ari}
\DeclareMathOperator{\arit}{arit}
\DeclareMathOperator{\Darit}{Darit}
\DeclareMathOperator{\lu}{lu}
\DeclareMathOperator{\Da}{Da}
\DeclareMathOperator{\da}{da}
\DeclareMathOperator{\swap}{\text{\textnormal{swap}}}
\DeclareMathOperator{\push}{\text{\textnormal{push}}}
\DeclareMathOperator{\ad}{\textnormal{ad}}
\newcommand{\MEM}{\mathsf{MEM}}
\newcommand{\MTM}{\mathsf{MTM}}
\DeclareMathOperator{\SL}{SL}

\newcommand{\uf}{\mathfrak{u}}
\renewcommand{\leq}{\leqslant}
\renewcommand{\geq}{\geqslant}
\renewcommand{\ge}{\geq}
\renewcommand{\le}{\leq}
\newcommand{\sing}{\textnormal{sing}}
\newcommand{\al}{\textnormal{al}}

\allowdisplaybreaks[3]

\newcommand\vartextvisiblespace[1][.5em]{%
  \makebox[#1]{%
    \kern.07em    \vrule height.3ex    \hrulefill    \vrule height.3ex    \kern.07em}}

\newcommand{\Rrel}[1]{\mathrm{R}_{#1}}
\newcommand{\swapbar}{\overline}

\begin{document}
\thispagestyle{empty}

\begin{abstract}
Richard Hain and Makoto Matsumoto constructed a category of universal
mixed elliptic motives, and described the fundamental~Lie algebra of this
category: it is a semi-direct product of the fundamental~Lie algebra
$\Lie\pi_1(\MTM)$ of the category of mixed Tate motives over~$\Z$ 
with a filtered and graded~Lie algebra~$\uf$. This~Lie algebra, and in
particular~$\uf$, admits a representation as derivations of the free~Lie
algebra on two generators.  In this paper we study the image~$\E$ of this
representation of~$\uf$, starting from some results~by Aaron~Pollack, 
who determined all the relations in a certain filtered quotient of~$\E$, 
and gave several examples of relations in low weights in~$\E$ that are 
connected~to period polynomials of
cusp forms on~$\SL_2(\Z)$.~Pollack's examples lead~to a conjecture on the 
existence of such relations in all depths and all weights, that we state in 
this article and prove in depth~$3$ in all weights.  The proof follows quite
naturally from Ecalle's theory moulds, to which we give a brief introduction. 
We prove two useful general theorems on moulds in the appendices.
\end{abstract}

\maketitle

%\setcounter{tocdepth}{1}
%\tableofcontents

\section{Introduction}

\subsection{Motivation}

In the unpublished text \cite{HainMatsumoto}, Hain and Matsumoto define a 
Tannakian category~$\MEM$ of mixed elliptic motives, and give a partially
conjectural description of its fundamental~Lie algebra. The elements of 
this~Lie algebra satisfy certain relations coming from modular forms that 
seem~to be related~to other natural appearances of modular forms in the theory 
of multiple (particularly double) zeta values; see \cite{IKZ,Speriodpolys} for the situation in the dual algebra, or ~\cite{HainToDeligne} for a 
cohomological explanation of this phenomenon.

Hain and Matsumoto show that the~Lie algebra~$\Lie\pi_1(\MEM)$ is a 
semi-direct product ~$\uf\rtimes \Lie\pi_1(\mathsf{\MTM})$, where the right-hand
factor is the~Lie algebra of the pro-unipotent radical of the fundamental group 
of the category of mixed Tate motives over~$\Z$ and~$\uf$ is a weight-graded 
Lie algebra equipped with a depth filtration, related~to
$\SL_2(\mathbb Z)$. 

Hain and Matsumoto construct a representation~$\Lie\pi_1(\MEM)\to\Der
\Lie[a,b]$. This representation is known~to be injective on the subalgebra 
$\Lie\pi_1(\MTM)$\footnote{This result follows from combining standard
results from the literature. Indeed, Goncharov constructed a Hopf algebra
${\mathcal{MZ}}$ of motivic multiple zeta values, and it follows from work
of F.~Brown (\cite{realBrown}) that $\Lie\pi_1(\MTM)$ is isomorphic to the dual of 
the quotient of ${\mathcal{MZ}}$ modulo products. Goncharov also showed that
elements of ${\mathcal{MZ}}$ satisfy the associator relations, which by 
duality gives an injective map $\Lie\pi_1(\MTM)
\hookrightarrow\grt$ into the Grothendieck-Teichm\"uller Lie algebra.
Finally, in \cite{Enriquez}, Enriquez constructed an explicit injective
map from $\grt$ into ${\mathrm{Der}}\,\Lie[a,b]$. The injectivity 
of the map $\Lie\pi_1(MTM)\hookrightarrow {\rm Der}\,\Lie[a,b]$
also appears as theorem 3.1 in the recent preprint \cite{Brown3}, where it has a very different 
proof based on substantial work by Hain and Levin-Racinet.} and conjectured 
to be injective on~$\uf$. The image 
of~$\uf$, denoted~$\E$, is equipped with a natural weight-grading and depth 
filtration, but is far from free.  It was studied~by Aaron~Pollack~\cite{Pollack},
 who defined a filtration~$\Theta$ on~$E$ different from the
depth filtration, and classified all relations in the quotient~$\E/\Theta^3\E$,
showing that in each weight~$n$ and depth~$d>1$, these relations are in
bijection with the period polynomials of the same parity as~$d$ associated~to 
modular forms on~$\SL_2(\Z)$.

This leads~to two natural questions. Firstly, one may ask whether 
Pollack's relations in~$\E/\Theta^3\E$ lift~to actual relations in the~Lie 
algebra~$\E$. ~Pollack computed several examples of such relations in low 
weights, and his observations on these examples lead~to a natural hypothesis
that we state as Conjecture~\ref{mainconj} at the end of this section.  
The conjecture is trivial in depth 2 since a depth 2 relation in 
$\E/\Theta^3\E$ also holds in~$\E$.  The goal of this paper is~to prove 
the conjecture in depth~$3$\footnote{The same statement appears without
proof in the online lecture notes \cite{Bnotes} by F. Brown.}.

The second natural question is, in the absence of knowledge about the 
injectivity of the map~$\uf\rightarrow\E$, whether~Pollack's relations in~$\E$ 
lift~to relations in~$\uf$, i.e. whether they are ``motivic''.  Hain recently
proved that~Pollack's depth 2 relations are motivic, but it remains an open
question for~$d\ge 3$.

The methods we use~to prove Conjecture~\ref{mainconj} are based on the
passage from non-com\-mu\-ta\-tive to commutative variables via techniques
from Ecalle's theory of moulds,~to which we give a very brief introduction
in section~\ref{sec3}.
The remarkable advantage of using moulds is that the proof involves 
not just polynomials such as elements of~$\Lie[a,b]$, but rational 
functions whose denominators play a very useful role.

\medskip
\noindent\textbf{Acknowledgments.} The first author thanks Cl\'ement~Dupont for
useful discussions. We also thank Richard~Hain for pointing out
an ambiguity in an early version of this paper, as well as the anonymous
referee for pointing out that the original proof was more complicated than necessary.

\subsection{The~Lie algebra~\texorpdfstring{$\E$}{E}: definition and statement of results}

The free~Lie algebra~$\Lie[a,b]$ is graded~by the \emph{weight}, i.e. the
degree of polynomials in~$a$ and~$b$, and~by the \emph{depth}, i.e. the
% depth of~$f\in\Lie[a,b]$ is equal~to the minimum number of times the 
% letter~$b$ occurs in any monomial of~$f$.  
minimum number of~$b$'s in any monomial.

The weight-grading induces a grading on the~Lie algebra~$\Der\Lie[a,b]$
of derivations of~$\Lie[a,b]$: a derivation is of weight~$n$ if the image 
of~$a$ and~$b$ are of weight~$n+1$. The depth-grading induces a grading 
on $\Der\Lie[a,b]$: a derivation~$D$ is of depth ~$r$ if~$D(a)$ is 
of homogeneous depth ~$r$ and~$D(b)$ of homogeneous depth ~$r+1$.

For any word~$w=w_1\cdots w_r$ with~$w_i\in \{a,b\}$, and any
$g\in \Lie[a,b]$, we write 
$w\cdot g=\ad(w_1)\cdots \ad(w_r)(g)$. Hain and Matsumoto 
showed that the image 
$\E$ of~$\uf$ is generated~by the derivations~$\eps_{2i}$ defined for 
$i\ge 0$~by
\[
  \eps_{2i}(a) = a^{2i}\cdot b\quad\text{and}\quad \eps_{2i}(b) =
\sum_{j=0}^{i-1}\,(-1)^j\, [a^{j}\cdot b,a^{2i-1-j}\cdot b].
\]
The~$\eps_{2i}$ all satisfy the relation~$\eps_{2i}([a,b])=0$. 
The derivation~$\eps_2$ commutes with all the others, so it plays no role
in our investigation of relations in~$\E$.

The relations between the brackets of derivations that are the subject of 
Pollack's work~\cite{Pollack} are more intricate than it might seem
from their simple definition.
In this article, we will concentrate on depth~$3$ relations, namely relations
between derivations of the form~$[\eps_{2i},[\eps_{2j},\eps_{2k}]]$ in~$\E$.

\medskip
Let~$\sl\subset \Der\Lie[a,b]$ be the~Lie algebra generated~by 
$\eps_0$ and a second derivation~$\phi_0$ defined~by~$\phi_0(a)=0$,
$\phi_0(b)=a$. Elements of~$\E$ that commute with~$\phi_0$ are called 
elements of {\it highest weight} of
$\E$.  The algebra~$\E$ is also equipped with a filtration denoted 
$\Theta$ (see~\hbox{\cite[p.~5--7]{Pollack}}), which is the filtration induced~by the 
descending central series filtration on the subalgebra of~$\Lie[a,b]$ 
generated~by the ~$w\cdot [a,b]$ with ~$w=a^ib^j\in\Q\chev{a,b}$.  

We can now give a precise formulation of~Pollack's main theorem. 
Let~$\E_2\subset\E$ be the subspace spanned~by elements of the form
\[[\eps_0^i\cdot\eps_p, \eps_0^j\cdot \eps_q],\ \ p,q>2\ {\rm even}.\]  
Then for all depths~$d>1$,~Pollack shows that the depth~$d$ elements

\[
 h_{p,q}^{d} =
\sum_{i+j=d-2}(-1)^i\tfrac{(d-2)!}{\binom{p}{i}\binom{q}{j}}\,[\eps_0^{i}\cdot
\eps_{p+2},\eps_{0}^{j}\cdot \eps_{q+2}].
\]
span the highest weight part of the subspace~$\E_2$.
Pollack's main theorem identifies the relations between these elements
in the quotient~$\E/\Theta^3\E$:

\begin{theo}[{\cite[Thm.~2]{Pollack}}] \label{theopollack}
Let~$D$ be a highest weight element of~$\E_2$, of weight~$n$ and 
depth~$d$. Then~$D\equiv 0\ [\Theta^3\E]$ if and only if
it is of the form 
\[\Rrel{f,d}= \sum_{p+q=n-4} r_{p-d+2}(f) \,h_{p,q}^d,\] 
where ~$r_i(f)$ is the ~$i$-th period~(in the sense of~\cite[\S\,11]{MR}) 
of a modular form~$f$ of weight~$n-2d+2$ on~$\SL_2(\Z)$.
\end{theo}

If~$f$ is a modular form of modular weight~$k$, then the 
element~$\Rrel{f,d}\in\E_2$ is of depth~$d$ and weight~$n=k+2d-1$.  The
depth~$2$ and~$3$ elements corresponding~to the weight 12 cusp form
known as the Ramanujan~$\Delta$ are given~by
\begin{gather*}
 \Rrel{\Delta,2}= h_{2,8}^{2} - 3\,h_{4,6}^{2}\equiv 0\ [\Theta^3\E] \\
 \text{and}\qquad\Rrel{\Delta,3}=4\,h_{2,10}^{3} - 25\,h_{4,8}^3 + 21\,h_{6,6}^3\equiv 0\ [\Theta^3\E].
\end{gather*}
%which correspond in depths~$2$ and~$3$ respectively~to the even and odd 
%period polynomials 
%\begin{gather*}
%\qquad\qquad X^8-3X^6+3X^4-X^2\\
 %\text{and}\qquad
%4X^9 - 25X^7+42X^5-25X^3+4X^1
%\end{gather*}
%of the weight 12 modular form known as the Ramanujan~$\Delta$.

Theorem~\ref{theopollack} shows that, as also occurs in other situations 
(cf. \cite{Ihara02,Speriodpolys,LNT,HainMatsumoto,Brown}) related~to the theory 
of multiple zeta values, periods of modular forms appear
as coefficients of relations in~$\E/\Theta^3\E$. ~Pollack asked the
following question: {\it when do these relations in~$\E/\Theta^3$ lift~to 
actual relations in~$\E$?}  Explicit calculation shows that this is not the 
case in general for the relations associated~to the Eisenstein series for 
even ~$d$: for example, the relation~$h_{d,n+d-2}^{d+2}\equiv 0\
[\Theta^3\E]$ associated~to the series ~$E_{n+4}$ does not lift~to an identity
in ~$\E$. However,~Pollack observed on several examples that the relations 
coming from cusp forms do seem~to lift~to relations in~$\E$. For example, 
the relation~$\mathrm{R}_{\Delta,3}\equiv 0\
[\Theta^3\E]$ admits the lifting
\begin{equation}\label{eqex}
 4\,h_{2,10}^{3} - 25\,h_{4,8}^3 + 21\,h_{6,6}^3 = -
\tfrac{345}{8}[\eps_6,[\eps_6,\eps_4]]+\tfrac{231}{20}[\eps_4,[\eps_8,\eps_2]]\text.
\end{equation}

We can thus formulate an explicit conjecture framing an answer~to~Pollack's 
question as follows.
\begin{conj}\label{mainconj}
Let~$k$ be an integer, and let ~$f$ be a cusp 
form of weight~$k$ for~$\SL_2(\Z)$. Then there exists a linear combination
$T$ of brackets of~$\eps_{2i}$, containing {\rm at least three}~$\eps_{2i}$
with~$i>1$, such that~$\Rrel{f,d}=T$.
\end{conj}

The goal of this article is~to prove this conjecture in the case~$d=3$.  The
theorem can be stated directly in this situation as follows.  

\begin{theo} Fix a positive integer~$n$, and let~$D\in\E$ be a linear 
combination of terms of the form~$\bigl[\eps_{2i},
[\eps_0,\eps_{2j}]\bigr]$ 
with~$i,j\ge 2$. Assume that~$D\in\Theta^3\E$.  Then~$D$ can be
written as a linear combination of terms of the form
$\bigl[\eps_{2r},[\eps_{2s},\eps_{2t}]\bigr]$ with~$r,s,t\ge 2$.
\end{theo}

We observe in passing that~$\eps_2$ never appears in any brackets
of~$\eps_{2i}$ because, as is easily seen~by hand, it is central in~$\E$.
Our proof is based on three ingredients: first of all, a reformulation in 
terms of polynomial algebra of the condition that a 
derivation in~$\E$ belongs~to~$\Theta^3\E$ (section~\ref{sec2}), secondly, a theorem 
due~to Goncharov characterizing certain types of~Lie elements in
depth~$3$ (end of section~\ref{sec2}), and thirdly, the passage~to 
Ecalle's language of moulds (section~\ref{sec3}) which allows us~to make interesting
use of rational functions and denominators. Two particularly useful theorems
from section~\ref{sec3}, one in mould theory and the other concerning the translation 
of the~Lie algebra~$\E$ into mould theory, are proved in appendices, and 
the proof of the main theorem is given in section~\ref{sec4}.

\section{A reformulation of~Pollack's property}\label{sec2}

\subsection{Properties of~Lie polynomials}
The following definition gives a key useful property of elements of~$\E$.

\begin{defi}
The endomorphism~$\push$ of the vector space~$\Q\chev{a,b}$ is defined~by 
its value on monomials, given~by
\[
 \push(a^{i_0}b\cdots a^{i_{r-1}}b\ a^{i_r}) = a^{i_r}ba^{i_0}b\cdots ba^{i_{r-1}}\text{.}
\]
A polynomial~$f$ is said~to be ~\emph{$\push$-invariant} if~$\push(f) = f$.
\end{defi}

It is known that~$\push$-invariant~Lie polynomials are exactly the values 
at~$a$ of derivations that are zero on~$[a,b]$; we recall this and another
useful characterization in the following proposition.

\begin{prop}[{\cite[Thm.~2.1]{LeilaKV}}]\label{LKV} Let~$P\in\Lie[a,b]$, and
assume that $P$ is of homogeneous degree $\ge 2$. \label{propL2i}
\begin{enumerate}[i)]\item\label{prop:LKV:i} The polynomial~$P$ is~$\push$-invariant if and only if
there exists an element~$Q$ in~$\Lie[a,b]$ such that~$[P,b]+[a,Q] = 0$, in
other words the derivation~$D$ defined~by~$D(a)=P$,~$D(b)=Q$ satisfies
$D([a,b])=0$. If such a~$Q$ exists then it is unique.
 \item\label{prop:LKV:ii} There exists a polynomial~$Q\in\Lie[a,b]$ such that~$P=[a,Q]$ if and only if~$P$ does not contain any monomials starting and ending in~$b$.
\end{enumerate}
\end{prop}

\begin{cor}\label{pushinv} Let~$D\in\E$.  Then~$D(a)$ is a~$\push$-invariant
polynomial.
\end{cor}

Indeed, since~$D\in\E$, we have~$D([a,b])=[a,D(b)]+[D(a),b]=0$, so
$D(a)$ is push-invariant~by Proposition~\ref{LKV}~\eqref{prop:LKV:i}.

We can now give the reformulation of~Pollack's property that will serve 
the purposes of our proof.

\begin{prop}\label{propP4}
Let~$D$ be a linear combination %in~$\E$ 
of terms of the form
$\bigl[\eps_{2i},[\eps_0,\eps_{2j}]\bigr]$ with~$i,j>1$.
If the derivation~$D$ lies in~$\Theta^3\E$,
there exists a~Lie polynomial~$Q\in\Lie[a,b]$ 
such that~$D(a)=[a,Q]$.
\end{prop}

\begin{proof}
By the definition of the filtration~$\Theta$, the derivation~$D$ lies
in ~$\Theta^3\E$ if and only if~$D(a)$ can be written as a linear 
combination of terms of the form~$[a^i\cdot b,[a^j\cdot b,a^k\cdot b]]$ with
~$i$,~$j$ and~$k$ strictly positive. Thus we can assume
\[D(a)=(a^i\cdot b)(a^j\cdot b)(a^k\cdot b)-(a^i\cdot b)(a^k\cdot b)(a^j\cdot b)
-(a^j\cdot b)(a^k\cdot b)(a^i\cdot b)+(a^k\cdot b)(a^j\cdot b)(a^i\cdot b).\]
Then the monomials ending with~$b$ arise only from the terms
\[(a^i\cdot b)(a^j\cdot b)a^kb-(a^i\cdot b)(a^k\cdot b)a^jb
-(a^j\cdot b)(a^k\cdot b)a^ib+(a^k\cdot b)(a^j\cdot b)a^ib.\]
In particular this shows that no monomial in~$D(a)$ can end in~$b^2$.
Thus since~$D(a)$ is a~$\push$-invariant polynomial~by Corollary~\ref{pushinv},
no monomial of~$D(a)$ can start and end with~$b$. Then~by Proposition 
\ref{LKV}~\eqref{prop:LKV:ii}, it follows that~$D(a)$ is of the form~$[a,Q]$.
\end{proof}

\subsection{A theorem of Goncharov}
We end this section~by giving a result of Goncharov which will be essential
to the proof of our main theorem.

\begin{defi}Set~$b_i=a^{i-1}b$ for~$i\ge 1$.  For every polynomial 
$P\in\Lie[a,b]$, let~$P_*$ be the polynomial obtained from~$P$~by
forgetting all the words ending in~$a$ and adding on the term
$\sum_{i\ge 1} \frac{(-1)^{i-1}}{i}(P|a^{i-1}b)\,b^i$, where~$(P|w)$ denotes the coefficient of
the word~$w$ in the expanded polynomial~$P$, where a Lie polynomial written
as a sum of Lie brackets is expanded as an ordinary polynomial in $a,b$ via
the rule $[a,b]=ab-ba$, with $Lie[a,b]$ considered as a subspace of its
universal enveloping algebra $\Q\langle a,b\rangle$.  The polynomial~$P$ is said~to satisfy 
the {\it linearized double shuffle relations} if, rewriting~$P_*$ as
a polynomial in the~$b_i$, it lies in~$\Lie[b_1,b_2,\ldots]$.
Following the notation of \cite{Brown}, we write~$\ls$ for the space of 
polynomials satisfying the linearized double 
shuffle relations. The space~$\ls$ is weight and depth graded, with
weight being the degree of the polynomials and depth the number of
$b$'s in each monomial.  We write~$\ls_n^d$ for the part of weight~$n$ and 
depth~$d$. These spaces are also studied in \cite{IKZ}, which uses the 
notation~$DSh_d(n-d)$ for~$\ls_n^d$.

To every~$P\in\Lie[a,b]$ we associate a derivation~$D_P$ of 
$\Lie[a,b]$~by setting~$D_P(a)=0$ and~$D_P(b)=[b,P]$. The~Lie bracket
on derivations defines another~Lie bracket
on the underlying vector space of~$\Lie[a,b]$, 
the {\it Poisson bracket}, denoted~$\{\cdot,\cdot\}$ and given explicitly~by
\[\{P,Q\}=[P,Q]+D_P(Q)-D_Q(P),\]
which is nothing other than an expression for the usual bracket of
derivations: namely, the identity~$[D_P,D_Q]=D_{\{P,Q\}}$ holds for every polynomials~$P$ and~$Q$.
\end{defi}

The following theorem was proven in \cite[\S~7]{Goncha} (but see also
\cite[7.3]{Brown} for a clearer explanation). The point that
will be essential in the proof of our main theorem is that the powers
of~$\ad(a)$ that appear in the statement are all~$\ge 2$.

\begin{theo}\label{Gonch} Let~$P\in\Lie[a,b]$ be a 
polynomial of homogeneous depth~$3$ satisfying the linearized double shuffle 
relations.  Then~$P$ is a linear combination of terms of the form
\[\bigl\{\ad(a)^{2r}(b),\{\ad(a)^{2s}(b),\ad(a)^{2t}(b)\}\bigr\}\]
with~$r,s,t\ge 2$.
\end{theo}

\section{Moulds}\label{sec3}

In this section we collect some of the basic results and definitions
of Ecalle's theory of moulds, which is the framework in which we will prove the main results. For a grand overview of Ecalle's theory containing the requisite
statements and definitions, see \cite{EcalleFlex}, and for a more detailed
introduction~to the part of mould theory specifically concerning multiple
zeta values and complete proofs of the basic results, see \cite{Sonline}.

\subsection{Definition}

A {\it mould} is a collection~$(A_r)_{r\ge 0}$ where each~$A_r$ is a function
of~$r$ commutative variables~$u_i$ over a given field, which here we
will take to be $\Q$. In particular, when $r=0$, $A_0(\emptyset)$ is a constant
in $\Q$.  The notation being redundant,
we will often write~$A(u_1,\ldots,u_r)$ rather than~$A_r(u_1,\ldots,u_r)$. We
will also consider moulds in variables~$v_i$. 

\medskip
Ecalle defines the following operations and properties:\nopagebreak
\begin{itemize}
\item the ari-bracket, a~Lie bracket on the space 
of moulds in the~$u_i$ (resp. in the~$v_i$) satisfying~$A(\varnothing)=0$.  
Under this bracket, said spaces become~Lie algebras denoted 
$\ARI$ (resp.~$\swapbar{\ARI}$)\footnote{%
Ecalle uses~$\ARI$ for {\it bimoulds},
which are functions of two sets of variables~$u_i$ and~$v_i$; in the framework
of this article we never use actual bimoulds (i.e. functions that are 
non-trivial in both families), but there are situations in which they are
very useful.%
}.  We write~$\ARI^{\pol}$ (resp.~$\swapbar{\ARI}^{\pol}$) 
for the spaces of polynomial-valued moulds; it follows directly from the 
definition of the~$\ari$-brackets on~$\ARI$ (resp.~$\swapbar{\ARI}$) that these 
subspaces are actually~Lie subalgebras. We recall the explicit expression for 
the~$\ari$-bracket in~Appendix~\ref{app:A} (see \cite{EcalleFlex},~\cite{RacinetArticle} or~\cite{Sonline});

\item the~$\swap$, an involutive variable-change map from~$\ARI$~to~$\swapbar{\ARI}$,
defined~by
\[\swap(A)(v_1,\ldots,v_r)=A(v_r,v_{r-1}-v_r,\ldots,v_1-v_2).\]
%If~$B=(B_r)_{r\ge 1}$ is a family of functions in the~$v_i$, then the inverse 
%\swap is defined~by
%\[\swap(B)(u_1,\ldots,u_r)=B(u_1+\cdots+u_r,u_1+\cdots+u_{r-1},\ldots,u_1).\]
From this expression it is clear that~$A$ is rational or polynomial-valued,
then so is~$\swap(A)$;

\item the~$\push$, a cyclic variable-change operator acting on moulds in~$\ARI$~by
\[\push(A)(u_1,\ldots,u_r)=A(-u_1-\cdots-u_r,u_1,\ldots,u_{r-1}).\]
We write~$\ARI_{\push}$ for the space of moulds that are invariant
under the~$\push$ operator;

\item the notion of alternality, a mould~$A$ in~$\ARI$ (resp. in~$\swapbar{\ARI}$) being said~to be
{\it alternal} if for each~$r>1$,~$A$ satisfies
\[
 \sum_{\substack{\sigma\in\mathfrak{S}_r \\ \sigma(1)<\cdots < \sigma(s) \\
\sigma(s+1)< \cdots < \sigma(r)\hphantom{{}+1}}}A(u_{\sigma^{-1}(1)},
\ldots,u_{\sigma^{-1}(r)}) =
0,
\]
(resp. the same condition with~$u_i$ replaced~by~$v_i$). Finally, we say that a 
mould is \emph{bialternal} if both~$A$ and~$\swap(A)$ are alternal. 
We write~$\ARI_{\al/\al}$ for the space of bialternal moulds.
\end{itemize}

\medskip
As can be seen in several recent articles~\cite{EcalleFlex,IKZ,Brown}, the
passage~to commutative variables can be very useful in studying
algebras that are described in terms of non-commutative variables.
By Lazard elimination, the Lie algebra ${\rm Lie}[a,b]$ can be written
as a direct sum
$${\rm Lie}[a,b]=\Q a\oplus {\rm Lie}[C_1,C_2,\ldots],$$
where~$C_i=\ad(a)^{i-1}(b)$ and the right-hand Lie algebra is the free
Lie algebra on the $C_i$.
Let $\Q\langle C_1,C_2,\ldots\rangle$ denote
the free non-commutative polynomial ring on the $C_i$,
and for each $r\ge 1$, let
$\Q_r\langle C_1,C_2,\ldots\rangle$ denote the subspace of 
$\Q\langle C_1,C_2,\ldots\rangle$ spanned by the depth 
$r$ monomials $C_{a_1}\cdots C_{a_r}$.
Define a linear map from $\Q_r\langle C_1,C_2,\ldots\rangle$ 
to $\Q[u_1,u_2,\ldots]$, where the~$u_i$ are commutative variables,~by
\[\ma:C_{i_1}\cdots C_{i_r} \mapsto (-1)^{i_1+\cdots+i_r-r}u_1^{i_1-1}\cdots u_r^{i_r-1}.\]
If~$P\in \Lie[C_1,C_2,\ldots]$, we write
$P=\sum_{r\ge 1} P^r$ where each~$P^r$
denotes the part of~$P$ that is homogeneous of depth~$r$; then~$\ma$
extends to a map~$\ma:\Lie[C_1,C_2,\ldots]\rightarrow \ARI$~by
taking~$\ma(P)$~to be the mould whose depth~$r$ part is given~by~$\ma(P^r)$.

\medskip
Under the map~$\ma$, properties of a~Lie polynomial translate into properties
of the associated mould.  We assemble the most useful ones
in the following theorem.
\vfill\eject
\begin{theo}\label{maprops}~
\begin{enumerate}[i)]
\item\label{th:maprops:1} The map~$\ma$ transports the Poisson bracket onto the 
$\ari$-bracket, that is~$\ma(\{f,g\})=\ari(\ma_f,\ma_g)$,
and restricts~to a~Lie algebra isomorphism from~$\Lie[a,b]$~to the subalgebra~$\ARI^{\pol}_{\al}$ of alternal 
polynomial-valued moulds in the~$u_i$. 
\item\label{th:maprops:2} A polynomial~$P\in\Lie[a,b]$ is~$\push$-invariant if and only
if~$\ma(P)$ is a~$\push$-invariant mould, and
the space~$\ARI_{\push}$ forms a~Lie algebra under the~$\ari$-bracket.
\item\label{th:maprops:3} Let~$Q\in\Lie[a,b]$ and~$P=[a,Q]$.  Then for each~$r\ge 1$, we have
\[\ma(P)(u_1,\ldots,u_r)=-(u_1+\cdots+u_r)\ma(Q)(u_1,\ldots,u_r)\text.\]
\item\label{th:maprops:4} The map~$\ma$ restricts~to an isomorphism
\[\ls\rightarrow \ARI^{\pol}_{\underline{\al}/\underline{\al}}\text,\]
of image the set of moulds in~$\ARI^{\pol}_{\al/\al}$ whose depth~$1$ part is an even function.
\item\label{th:maprops:5} The space~$\ARI_{\underline{\al}/\underline{\al}}$ is a~Lie algebra under 
the~$\ari$-bracket, and is contained in~$\ARI_{\push}$.
\end{enumerate}
\end{theo}

All these results constitute a standard part of mould theory, and can 
be found scattered through various texts~(see~\cite{EcalleFlex} for the statements
not concerning~$\Lie[a,b]$, or~\hbox{\cite{RacinetArticle,Baumard}}; 
see \cite{Sonline} for a basic introduction~to the domain containing complete
proofs.)

\begin{defi}For each~$r\ge 1$, let 
\[\De_r(u_1,\ldots,u_r)=u_1\cdots u_r(u_1+\cdots+u_r).\]
In particular~$\De_1(u_1)=u_1^2$. Let~$\ARI^{\sing}$ denote the space 
of rational-function valued moulds~$A$ such that~$\De_rA_r(u_1,\ldots,u_r)$
is polynomial for each~$r\ge 1$.  
\end{defi}

The first main result of this article is the following theorem, which is
a key result in the application of mould theory~to elliptic motives.
Its rather lengthy proof is deferred~to~Appendix~\ref{app:A}.

\begin{theo}[{\cite[th.~4.45]{Baumard}}]\label{these}The space 
\[\ARI_{\underline{\al}/\underline{\al}}^{\sing}=
\ARI_{\underline{\al}/\underline{\al}}\cap \ARI^{\sing}\]
is a~Lie algebra under the ari-bracket.
\end{theo}

\vspace{.2cm}
\noindent {\bf Remark.} It is shown in \cite{Sonline}, Corollary 3.3.4, that 
the $\ari$-bracket restricted to {\it polynomial-valued} moulds corresponds to 
the Poisson (or Ihara) bracket.  Therefore it follows from (a graded version
of) Racinet's famous theorem (cf. \cite{RacinetArticle}) on the Lie algebra 
structure of the double shuffle Lie algebra that the
linearized double shuffle space is also Lie algebra under the Poisson bracket,
or equivalently, in mould language, that the space of polynomial bialternal 
moulds, $\ARI_{\underline{\al}/\underline{\al}}^{\pol}$, forms a Lie algebra 
under the $\ari$-bracket.  But Theorem \ref{these} above generalizes this 
statement
to the much larger space $\ARI_{\underline{\al}/\underline{\al}}^{\sing}$, 
which contains moulds with denominators of type $\De_r$, so the key new point 
(cf. Appendix) is the proof that the $\ari$-bracket of two such moulds has the 
same type of denominator.

\begin{defi}
For~$i\ge -1$, let~$U_{2i}$ be the mould such that 
$U_{2i}(u_1)=u_1^{2i}$ in depth~$1$, and~$U_{2i}(u_1,\ldots,u_r)=0$
for every depth~$r\ne 1$ (we say that~$U_{2i}$ is concentrated
in depth~$1$).  Let 
$\mathcal{U}$ be the~Lie subalgebra of~$\ARI$ generated~by the~$U_{2i}$. 
Since~$\De_1=u_1^2$, we have~$U_{2i}\in \ARI^{\sing}$ for~$i\ge -1$, so thanks
to~Theorem~\ref{these}, we have the inclusion~$\mathcal{U}\subset 
\ARI_{\underline{\al}/\underline{\al}}^{\sing}$.
\end{defi}

Let~$\Der^0\Lie[a,b]$ denote the subspace of
$\Der\Lie[a,b]$ consisting of the derivations that kill~$[a,b]$, and such
that the value of the derivation on $a$ has no linear term in $a$.
For every polynomial~$F\in\Lie[a,b]$, let 
$\Da(F)$ denote the mould which is given in depth~$r$~by
\[\Da(F)(u_1,\ldots,u_r)=(1/\De_r)\,\ma(F)(u_1,\ldots,u_r).\]
Define a map
$\Psi:\Der^0\Lie[a,b]\rightarrow \ARI$~by
\[\Psi(D)=\Da\bigl(D(a)\bigr).\]

The following proposition gives the key relationship between brackets
of derivations of~$\Lie[a,b]$ and the~$\ari$-bracket on moulds.

\begin{theo}\label{B3} 
The map~$\Psi$ is an injective~Lie algebra morphism, i.e.
\[\Psi\bigl([D_1,D_2]\bigr)\mapsto \ari\Bigl(\Da\bigl(D_1(a)\bigr),
\Da\bigl(D_2(a)\bigr)\Bigr).\]
\end{theo}

The rather long proof of this theorem is deferred~to~Appendix~\ref{app:B}.  Note
however that the injectivity is easy, since~$D(a)$ can be recovered from
$\Da(D(a))$, and it follows from Proposition~\ref{LKV}
that a derivation in~$\Der^0\Lie[a,b]$ is determined~by its
value on~$a$.
 
The next result\footnote{This result is also given
in \S 4 of the preprint \cite{newEnriquez}~by B. Enriquez from January 2013, 
with a brief indication in lieu of proof.  Brown (cf \cite{Bnotes}, \cite{Brown3}) studies a Lie algebra denoted
$\pls$ (polar linearized double shuffle) which is isomorphic to
$\swap\Bigl(\ARI^{\sing}_{\al/\al}\Bigr)$.  In the recent preprint 
\cite{Brown3}, Brown mentions without proof the analogous result to Theorem 
\ref{these} for $\pls$ (see Definition 4.5), i.e.~that $\pls$ is closed under 
the ``linearized Ihara bracket''. He also gives results that are analogous to 
Theorem \ref{B3} and Corollary 3.6.}
uses this proposition~to show that~$\mathcal{U}$ is 
isomorphic~to~Pollack's~Lie algebra~$\E$.  
\begin{cor}[\cite{newEnriquez}]\label{Psi}
We have~$\Psi(\eps_{2i})=U_{2i-2}\in\mathcal{U}$ for all~$i\ge 0$, and 
the map~$\Psi$ induces an isomorphism~$\E\simeq \mathcal{U}$ of~Lie algebras.
\end{cor}

\begin{proof} Since~$\eps_{2i}(a)=\ad(a)^{2i}(b)=C_{2i+1}$, the mould
$\ma(\eps_{2i})$ is the mould in depth~$1$ that takes
the value~$u_1^{2i}$, so~$\Psi(\eps_{2i})=\Da(D(a))=\Da(C_{2i+1})
=U_{2i-2}$.  Then~by~Theorem~\ref{B3}, the map~$\Psi$ 
restricted~to the~Lie algebra 
$\E\subset \Der^0\Lie[a,b]$ generated~by the~$\eps_{2i}$ 
yields a~Lie algebra isomorphism~to~$\mathcal{U}$.\end{proof}

%The above results imply Lemma~\ref{pushinv} from section~\ref{sec2}, as follows.
%
%\begin{cor}\label{Lemma23}Let~$D\in\E$.  Then~$D(a)$ is a~$\push$-invariant polynomial.
%\end{cor}
%
%\begin{proof} For each~$i\ge -1$, we have~$\Psi(\eps_{2i})=u_1^{2i-2}$.
%By linearity, we may assume that~$D$ is a bracket of~$r$~$\eps_{2i}$'s,
%with~$r\ge 1$.  Being even functions concentrated in depth~$1$, the~$U_{2i}$
%are all~$\push$-invariant moulds, so~by (3) of~Theorem~\ref{maprops}, every
%element of~$\mathcal{U}$ is~$\push$-invariant. In particular since
%$\Psi(D)\in \mathcal{U}$~by Corollary~\ref{Psi},
%$\Psi(D)$ is~$\push$-invariant.  Since~$\De_r(u_1,\ldots,u_r)$ is
%$\push$-invariant for each~$r\ge 1$, the mould~$\ma(D(a))$ is 
%also~$\push$-invariant, which means that the polynomial~$ma(D(a))$ is 
%$\push$-invariant, which~by~Theorem~\ref{maprops} (2) is equivalent~to 
%the~$\push$-invariance of the polynomial~$D(a)$.
%\end{proof} 

\section{Proof of the main result}\label{sec4}

We begin~by translating~Theorem~\ref{Gonch}~to a statement on moulds.

\begin{theo}\label{mouldGonch}Suppose~$A$ is a bialternal 
polynomial-valued mould concentrated in depth~$3$. Then~$A$ lies in the~Lie algebra~$\mathcal{U}$, and more precisely it can be written
as a linear combination of moulds of
the form~$\ari\bigl(U_{2r},\ari(U_{2s},U_{2t})\bigr)$ with~$r,s,t\ge 1$.
\end{theo}

Indeed,~by~\eqref{th:maprops:1} of~Theorem~\ref{maprops}, the map ma gives an isomorphism
from the space of depth~$3$~Lie polynomials in~$\ls$~to the space of
depth~$3$ moulds in~$\ARI^{\pol}_{\al/\al}$, and~by~\eqref{th:maprops:4} of the same~Theorem, 
since~$\ma\bigl(\ad(a)^{2r+2}(b)\bigr)=U_{2r}$, we have
\[\ma\Bigl(\bigl\{\ad(a)^{2r+2}(b),\{\ad(a)^{2s+2}(b),\ad(a)^{2t+2}(b)\}\bigr\}
\Bigr) =\ari(U_{2r},\ari(U_{2s},U_{2t})),\] so the statement
of~Theorem~\ref{mouldGonch} is equivalent~to that of~Theorem~\ref{Gonch}.

We can now prove the main result of this article.

\begin{theo} Fix a positive integer~$n$, and let~$D\in\E$ be a linear 
combination of terms of the form~$\bigl[\eps_{2i},
[\eps_0,\eps_{2j}]\bigr]$ 
with~$i,j\ge 2$. Assume that~$D\in\Theta^3\E$.  Then~$D$ can be
written as a linear combination of terms of the form
$\bigl[\eps_{2r+2},[\eps_{2s+2},\eps_{2t+2}]\bigr]$ with~$r,s,t\ge 1$.
\end{theo}

\begin{proof} By Proposition~\ref{propP4}, since~$D\in\Theta^3\E$, there 
exists a polynomial~$Q\in\Lie[a,b]$ such that~$D(a)=[a,Q]$. By~\eqref{th:maprops:3}
of~Theorem~\ref{maprops}, this means that we have the following relation
between the depth~$3$ moulds~$\ma(D(a))$ and~$\ma(Q)$:
\[-(u_1+u_2+u_3)\ma(Q)=\ma(D(a)).\]
Thus~$\ma(D(a))$ is a polynomial divisible~by the factor~$(u_1+u_2+u_3)$,
and since the mould~$\ma(D(a))$ is~$\push$-invariant~by Lemma~\ref{pushinv} and~\eqref{th:maprops:2} of~Theorem~\ref{maprops}, it is also divisible~by~$u_1$,~$u_2$ and~$u_3$.
Thus in fact~$\ma(D(a))$ is divisible~by~$\De_3$, so~$\Psi(D)=\Da\bigl(D(a)\bigr)$ is a polynomial-valued mould in 
$\mathcal{U}$ concentrated in depth~$3$.  Now, we saw that~$\mathcal{U}$ is contained in~$\ARI^{\pol}_{\underline{\al}/\underline{\al}}$,~i.e. every mould in~$\mathcal{U}$ is bialternal.  In particular, the depth~$3$ polynomial 
mould~$\Psi(D)\in\mathcal{U}$ is bialternal.  
By~Theorem~\ref{mouldGonch}, we can write it as
\[\Psi(D)(u_1,u_2,u_3)
=\sum_{r,s,t\ge 1} c_{rst}\,\ari\bigl(U_{2r},\ari(U_{2s},U_{2t})\bigr).\] 
Then since~$\Psi(\eps_{2i})=U_{2i-2}$ and~$\Psi:\E\rightarrow\mathcal{U}$ 
is a~Lie algebra isomorphism~by Corollary~\ref{Psi}, if we set
\[H=\sum_{r,s,t\ge 1} c_{rst}\,\bigl[\eps_{2r+2},[\eps_{2s+2},\eps_{2t+2}]\bigr]\]
we must have~$\Psi(D)=\Psi(H)$, so~by the injectivity of~$\Psi$,
we have~$D=H$, proving the theorem.
\end{proof}

\appendix\numberwithin{equation}{section}
\section{Proof of~Theorem~\ref{these}}\label{app:A}

This~appendix is devoted~to the proof of~Theorem~\ref{these}, which makes
essential use of the~$\swap$ operator and some of the basic notions of
Ecalle's theory of moulds, in particular the~$\ari$-bracket.

\begin{defis}
If we break the tuple~$w=(u_1,\ldots,u_r)$
into three parts 
\[a=(u_1,...,u_k),\ \ b=(u_{k+1},...,u_{\ell}),
\ \ c=(u_{\ell+1},...,u_r),\]
we write
\[\begin{cases}
a\lceil c=(u_1,\ldots,u_k,u_{k+1}+\cdots+u_{\ell}+u_{\ell+1},u_{\ell+2},\ldots,u_r)&
\textnormal{if }c\ne\varnothing%,{\rm \ otherwise\ } a\lceil c=a
\\
a\rceil c=(u_1,\ldots,u_k+u_{k+1}+\cdots+u_{\ell},u_{\ell+1},\ldots,u_r)
&\textnormal{if }a\ne\varnothing%,{\rm \ otherwise\ } a\rceil c=c
\end{cases}\]
If~$A,B\in \ARI$, we define the operator~$\arit(B)$ acting on~$A$~by 
\[\bigl(\arit(B)\cdot A\bigr)(u_1,\ldots,u_r)=\sum_{0\le k<\ell<r} A(a\lceil c)B(b)
-\sum_{1\le k<\ell\le r} A(a\rceil c)B(b).\]
If~$A,B\in\swapbar{\ARI}$ and the tuple~$(v_1,\ldots,v_r)$ breaks into pieces
\[a=(v_1,...,v_k),\ \ b=(v_{k+1},...,v_{\ell}), \ \ c=(v_{\ell+1},...,v_r),\]
then setting
\[\begin{cases}b\rfloor =(u_{k+1}-u_{\ell+1},u_{k+2}-u_{\ell+1},\ldots,u_{\ell}-u_{\ell+1})
&\textnormal{if }c\ne\varnothing,{\rm \ otherwise\ } b\rfloor =b\\
\lfloor b=(u_{k+1}-u_k,u_{k+2}-u_k,\ldots,u_{\ell}-u_k)
&\textnormal{if }a\ne\varnothing,{\rm \ otherwise\ } \lfloor b =b,
\end{cases}\]
we let
\[\bigl(\arit(B)\cdot A\bigr)(v_1,\ldots,v_r)=\sum_{0\le k<\ell<r} 
A(ac)B(b\rfloor) -\sum_{1\le k<\ell\le r} A(ac)B(\lfloor b).\]

\medskip
For two moulds~$A$ and~$B$, let~$\muu(A,B)$ be the product defined~by
\[\muu(A,B)(u_1,\ldots,u_r)=\sum_{i=0}^r A(u_1,\ldots,u_i)B(u_{i+1},\ldots,u_r).\]
When~$A=\ma(F)$ and~$B=\ma(G)$ for
polynomials~$F,G$ in~$a,b$, the multiplication~$\muu$ coincides with 
ordinary multiplication of polynomials:~$\muu(\ma(F),\ma(G))=\ma(FG)$~(cf.~\cite[(3.2.13)]{Sonline}). 
Let~$\lu(A,B)=\muu(A,B)-\muu(B,A)$, so~$\lu$ is the corresponding~Lie bracket.
For~$A,B\in \ARI$ or~$\swapbar{\ARI}$, we set 
\begin{equation}\label{eq:A1}
\ari(A,B)=\arit(B)\cdot A-\arit(A)\cdot B+\lu(A,B)\text.
\end{equation}

\medskip
Recall that~$\De$ is the mould in the~$u_i$'s defined~by
\[\De(u_1,\ldots,u_r)=u_1\cdots u_r(u_1+\cdots+u_r).\]
Let~$\De_v=\swap(\De)$, explicitly
\[\De_v(v_1,\ldots,v_r)=v_1(v_1-v_2)\cdots (v_{r-1}-v_r)v_r.\] 
For any mould~$A\in \swapbar{\ARI}$, let~$\Check{A}$ be the mould defined~by
\[\Check{A}(v_1,\ldots,v_r)=\De_v(v_1,\ldots,v_r)A(v_1,\ldots,v_r).\]
Let~$\swapbar{\ARI}^{\sing}_{\al}$ denote
the space of alternal rational-valued moulds~$A$ in the variables 
$v_i$ such that~$\Check{A}$ is polynomial-valued.  
\end{defis}

The heart of the proof of~Theorem~\ref{these}~\hbox{\cite[lemme~4.40]{Baumard}} consists in the following 
proposition.

\begin{prop}\label{prop:a1}
 The space
$\swapbar{\ARI}^{\sing}_{\al}$ is a~Lie algebra under the~$\ari$-bracket.
\end{prop}

We first need a lemma.

\begin{lem}\label{lem:a3}
 Let~$M$ be an element of~$\swapbar{\ARI}^{\sing}_{\al}$. Then the mould~$\Check M$ 
satisfies the relation
\[\Check{M}(0,v_2,\ldots,v_{r}) = \Check{M}(v_2,\ldots,v_r,0)\text.\]
\end{lem}

\begin{proof}[Proof of Lemma~\ref{lem:a3}]
Let us write the simplest of the alternality relations, which gives

\begin{align*}
 0 &= \sum_{w\in\sham(v_1,v_2\cdots v_r)}M(w) \\
&= \sum_{i=0}^{r}M(v_2,\ldots,v_i,v_1,v_{i+1},\ldots,v_r) \\
&= \sum_{i=1}^{r-1}\frac{\Check{M}(v_2,\ldots,v_i,v_1,v_{i+1},\ldots,v_r)}{v_2\,(v_2-\cdots)\cdots(\cdots-v_r)\,v_r} 
+\nline+
\frac{\Check{M}(v_1,\ldots,v_r)}{v_1\,(v_1-v_2)\cdots (v_{r-1}-v_r)\,v_r }
+\nline+
\frac{\Check{M}(v_2,\ldots,v_{r},v_1)}{v_2\,(v_2-v_3)\cdots (v_{r-1}-v_{r})(v_r-v_1)\,v_1}.\\
\intertext{Hence, multiplying~by~$v_1$ and evaluating at~$v_1=0$,}
 0 &= \frac{\Check{M}(0,v_2,\ldots,v_r)}{(-v_2)(v_2-v_3)\cdots (v_{r-1}-v_r)\,v_r }
+
\frac{\Check{M}(v_2,\ldots,v_{r},0)}{v_2\,(v_2-v_3)\cdots (v_{r-1}-v_{r})v_r},
\end{align*}
which is well-defined since both numerators at play are polynomials~by the hypothesis on~$M$. The result follows immediately.\end{proof}

\begin{proof}[Proof of Proposition~\ref{prop:a1}]
% We can now proceed~to the proof of Proposition~\ref{prop:a1}.
Let~$M,N\in \swapbar{\ARI}^{\sing}_{\al}$. By linearity, we may assume that~$M$ and~$N$
are respectively concentrated in depths~$r\geq 1$ and~$s\geq 1$.
Now, alternality is preserved~by the~$\ari$-bracket~(cf. \cite[Theorem 3.1]{SS} 
for a complete proof), so it only remains to show that~$\ari(M,N)$ is 
an element of~$\swapbar{\ARI}^{\sing}$.

The essential step of the proof consists in determining the poles of~\hbox{$\arit(M)\cdot N$} by reducing their values 
modulo the subspace of polynomials in the~$v_i$'s. We will use the simplifying
notation~$\De_v(v_1,\ldots,v_r)=\De_r(v)$. The fact that~$M$ is
concentrated in depth~$r$ and~$N$ in depth~$s$ also simplifies the defining
formula for~$\arit(M)\cdot N$, as follows:
\begin{align*}
\bigl(\arit(M)&\cdot N\bigr)(v_1\ldots v_{r+s})= \sum_{0\leq i< s}N(v_1\cdots\,v_i\, v_{i+r+1}\cdots\,v_{r+s})\,M(v_{i+1}\cdots 
v_{i+r}\rfloor)-\\
&\qquad\qquad\qquad\qquad\qquad \sum_{0< i\leq s}N(v_1\cdots\,v_i\, v_{i+r+1}\ldots v_{r+s})\,M(\lfloor v_{i+1}\cdots\,v_{i+r})\\
&= N(v_{r+1}\cdots\,v_{r+s})\,M(v_1-v_{r+1},\ldots,v_r-v_{r+1})\\ 
&\ \ \ +\sum_{i=1}^{s-1} N(v_1\cdots\,v_i\, v_{i+r+1}\cdots 
v_{r+s})\cdot\\
&\qquad\Bigl(M(v_{i+1}-v_{i+r+1},\ldots,v_{i+r}-v_{i+r+1}) - M(v_{i+1}-v_i,\ldots,v_{i+r}-v_i)\Bigr)\\
&\qquad\ \ - N(v_1\cdots\,v_s)\,M(v_{s+1}-v_s,\ldots v_{r+s}-v_s).
\end{align*}

\medskip
\noindent Consequently,
\begin{align}
&\De_{r+s}(v)\bigl(\arit(M)\cdot N\bigr)(v_1,\ldots ,v_{r+s})=\nonumber\\
&\ \ \  \tfrac{v_1}{v_{r+1}\,(v_1-v_{r+1})}\,\Check N(v_{r+1}\cdots 
v_{r+s})\,\Check M(v_1-v_{r+1},\ldots,v_r-v_{r+1}) 
+\sum_{i=1}^{s-1} \den{r+s}{v}\, S_i \nonumber\\
&\ \ \quad + \tfrac{v_{r+s}}{v_s\,(v_{r+s}-v_s)}\,\Check N(v_1\cdots\,v_s)\,\Check M(v_{s+1}-v_s,\ldots, 
v_{r+s}-v_s) 
\label{eq:A2}
% \qquad\qquad\qquad\qquad\ \ \ {(A.2)}
\end{align}
with
\begin{align*}
 \den{r+s}{v}\, S_i &= \frac{(v_i-v_{i+1})\cdots (v_{i+r}-v_{i+r+1})}{v_i-v_{i+r+1}}\,\Check{N}(v_1\cdots\,v_i\, v_{i+r+1}\cdots\,v_{r+s})\cdot\\
&\ \ \ \Bigl(M(v_{i+1}-v_{i+r+1},\ldots,v_{i+r}-v_{i+r+1}) -
M(v_{i+1}-v_i,\ldots,v_{i+r}-v_i)\Bigr)\\
&= \frac{(v_i-v_{i+1})\cdots (v_{i+r}-v_{i+r+1})}{v_i-v_{i+r+1}}\,\Check{N}(v_1\cdots\,v_i\, v_{i+r+1}\cdots v_{r+s})\cdot\\
&\ \ \ \Biggl(\frac{\Check 
M(v_{i+1}-v_{i+r+1},\ldots,v_{i+r}-v_{i+r+1})}{(v_{i+1}-v_{i+r+1})(v_{i+1}-v_{i+2})\cdots 
(v_{i+r-1}-v_{i+r})(v_{i+r}-v_{i+r+1})} -\\
&\ \ \ \ \ \qquad \frac{\Check 
M(v_{i+1}-v_i,\ldots,v_{i+r}-v_i)}{(v_{i+1}-v_i)(v_{i+1}-v_{i+2})\cdots(v_{i+r-1}-v_{i+r})(v_{i+r}-v_i)}\Biggr),
\end{align*}
or, equivalently,
\begin{align} 
\den{r+s}{v}\, S_i&= \tfrac{1}{v_i-v_{i+r+1}}\,\Check{N}(v_1\cdots\,v_i\, v_{i+r+1}\cdots\,v_{r+s})\cdot\nonumber\\
&\ \ \ \ \ \Bigl( \tfrac{v_i-v_{i+1}}{v_{i+1}-v_{i+r+1}}\,\Check 
M(v_{i+1}-v_{i+r+1},\ldots,v_{i+r}-v_{i+r+1})\nonumber\\
&\ \quad \quad \qquad\quad+\tfrac{v_{i+r}-v_{i+r+1}}{v_{i+r}-v_i}\,\Check 
M(v_{i+1}-v_i,\ldots,v_{i+r}-v_i)\Bigr)\text.
\label{eq:A3}
% \qquad\qquad\ \ \ \ {(A.3)}
\end{align}
In light of equalities~\eqref{eq:A2} and~\eqref{eq:A3}, there are three types of poles which 
can occur in the rational function~$\den{r+s}{v} \bigl(\arit(M)\cdot N\bigr)$:
\begin{samepage}\label{page:discussion}
\begin{enumerate}[1)]
 \item\label{cas1} the poles~$\frac1{v_{r+1}}$ and~$\frac{1}{v_{s}}$, which come from only one term and thus do not immediately cancel out;
  \item\label{cas3} the poles of the form~$\tfrac{1}{v_i-v_{i+r+1}}$;
\item\label{cas2} the poles of the form~$\tfrac{1}{v_i-v_{i+r}}$.
\end{enumerate}
\end{samepage}
Let us deal first with case~\ref{cas3}. The corresponding pole~$\tfrac{1}{v_i-v_{i+r+1}}$ can only appear in the term~$\den{r+s}{v}\,S_i$; we then want~to check that this pole is compensated~by the corresponding difference of the~$\Check M$'s. Let us compute this difference for~$v_i=v_{i+r+1}=u_i=u_{i+r+1}=a$; it gives
\[\tfrac{a-v_{i+1}}{v_{i+1}-a}\,\Check M(v_{i+1}-a,\ldots,v_{i+r}-a)+%\nlineshort+
\tfrac{v_{i+r}-a}{v_{i+r}-a}\,\Check 
M(v_{i+1}-a,\ldots,v_{i+r}-a)
\]
which is indeed zero, hence the compensation.

\medskip
For case~\ref{cas2}, we need~to distinguish three sub-cases, according~to whether~$i=1$, or~$i=s$, or~$i$ 
belongs~to~$\seg{2,s-1}$.
\begin{enumerate}[a)]
 \item~$i=1$: the pole is multiplied~by
\begin{align*}
 & \tfrac{v_1}{v_{r+1}}\,\Check N(v_{r+1}v_{r+2}\cdots\,v_{r+s})\,\Check M(v_1-v_{r+1},\ldots,v_r-v_{r+1}) 
-\nline-\tfrac{1}{v_1-v_{r+2}}\,\Check{N}(v_1\, v_{r+2}\cdots\,v_{r+s})\cdot (v_{r+1}-v_{r+2})\, \Check 
M(v_{2}-v_1,\ldots,v_{r+1}-v_1)
\end{align*}
which if~$v_1=v_{r+1}=a$ gives
\begin{align*}
 & \tfrac{a}{a}\,\Check N(a,v_{r+2},\ldots, v_{r+s})\,\Check M(a-a,v_2-a,\ldots,v_r-a) 
-\nline-\tfrac{1}{a-v_{r+2}}\,\Check{N}(a, v_{r+2},\cdots, v_{r+s})\cdot (a-v_{r+2})\, \Check M(v_{2}-a,\ldots,v_{r}-a,a-a) 
\\
&= \Check N(a,v_{r+2},\ldots, v_{r+s})\,\Check M(0,v_2-a,\ldots,v_r-a) -\nline-\Check{N}(a, v_{r+2},\cdots, v_{r+s})\,  
\Check M(v_{2}-a,\ldots,v_{r}-a,0)
\end{align*}
which is zero thanks~to Lemma~\ref{lem:a3}.
 \item~$i=s$: the pole is multiplied~by
\begin{align*}
 & -\tfrac{v_{r+s}}{v_s}\,\Check N(v_1\cdots\,v_s)\,\Check M(v_{s+1}-v_s,\ldots, v_{r+s}-v_s)+\nlineshort+
\tfrac{1}{v_{s-1}-v_{r+s}}\,\Check{N}(v_1,\ldots, v_{s-1}, v_{r+s})\cdot(v_{s-1}-v_{s})\, \Check 
M(v_{s}-v_{r+s},\ldots,v_{r+s-1}-v_{r+s})
\end{align*}
which if~$v_s=v_{r+s}=a$ gives
\begin{align*}
 & -\tfrac{a}{a}\,\Check N(v_1,\ldots,v_{s-1}, a)\,\Check M(v_{s+1}-a,\ldots, v_{r+s-1}-a,a-a)+\nline+
\tfrac{1}{v_{s-1}-a}\,\Check{N}(v_1,\ldots, v_{s-1}, a)\cdot(v_{s-1}-a)\, \Check M(a-a,v_{s+1}-a,\ldots,v_{r+s-1}-a) \\
&= -\Check N(v_1,\ldots,v_{s-1}, a)\,\Check M(v_{s+1}-a,\ldots, v_{r+s-1}-a,0)+\nline+
\Check{N}(v_1,\ldots, v_{s-1}, a)\Check M(0,v_{s+1}-a,\ldots,v_{r+s-1}-a) 
\end{align*} 
which is zero for the same reason as above.
 \item~$i\in\seg{2,s-1}$: the pole~$\tfrac{1}{v_i-v_{i+r}}$ comes from~$\den{r+s}{v} (S_{i-1}-S_i)$, and is multiplied~by
\begin{align*}
  &\tfrac{v_{i-1}-v_{i}}{v_{i-1}-v_{i+r}}\,\Check{N}(v_1,\ldots, v_{i-1}, v_{i+r},\cdots, v_{r+s})\cdot
\Check M(v_{i}-v_{i+r},\ldots,v_{i-1+r}-v_{i+r})-\nline- \tfrac{v_{i+r}-v_{i+r+1}}{v_i-v_{i+r+1}}\,\Check{N}(v_1,\ldots, 
v_i, v_{i+r+1},\ldots, v_{r+s})\cdot\Check M(v_{i+1}-v_i,\ldots,v_{i+r}-v_i)
\end{align*}
which if~$v_i=v_{i+r}=a$ gives
\begin{align*}
  &\Check{N}(v_1,\ldots, v_{i-1}, a,v_{i+r+1},\cdots, v_{r+s})\cdot
\Check M(a-a,v_{i+1}-a,\ldots,v_{i-1+r}-a)-\nline- \Check{N}(v_1,\ldots,v_{i-1}, a, v_{i+r+1},\ldots, v_{r+s})\cdot\Check 
M(v_{i+1}-a,\ldots,v_{i+r-1}-a,a-a) \\
&= \Check{N}(v_1,\ldots, v_{i-1}, a,v_{i+r+1},\cdots, v_{r+s})\cdot
\Check M(0,v_{i+1}-a,\ldots,v_{i-1+r}-a)-\nline- \Check{N}(v_1,\ldots,v_{i-1}, a, v_{i+r+1},\ldots, v_{r+s})\cdot\Check 
M(v_{i+1}-a,\ldots,v_{i+r-1}-a,0)
\end{align*}
which is zero again.
\end{enumerate}

\medskip
In consequence, the only remaining poles in~$\den{r+s}{v} \bigl(\arit(M)\cdot 
N\bigr)$ are those in~$\frac{1}{v_{r+1}}$ and~$\frac{1}{v_s}$ from case~\ref{cas1}, and more precisely we can write
\begin{align*}
\den{r+s}{v} &\bigl(\arit(M)\cdot N\bigr)(v_1,\ldots, v_{r+s})\equiv\\
&\tfrac{v_1}{v_{r+1}\,(v_1-v_{r+1})}\,\Check 
N(v_{r+1},\ldots, v_{r+s})\,\Check M(v_1-v_{r+1},\ldots,v_r-v_{r+1})\\ 
&+
\tfrac{v_{r+s}}{v_s\,(v_{r+s}-v_s)}\,\Check N(v_1,\ldots, v_s)\,\Check M(v_{s+1}-v_s,\ldots v_{r+s}-v_s) 
\pmod{\Q[\underline{v}]}\text.
 \end{align*}
It remains~to show that the above expression cancels in the bracket~$\ari(M,N)$,
using the definition by equation~\eqref{eq:A1}.  The possible poles for the rational function
$\den{r+s}{v}\ari(M,N)$ come from the sum
\begin{align*}
 &\tfrac{v_1}{v_{s+1}\,(v_1-v_{s+1})}\,\Check M(v_{s+1},\ldots, v_{s+r})\,\Check N(v_1-v_{s+1},\ldots,v_s-v_{s+1}) +\nline+
\tfrac{v_{s+r}}{v_r\,(v_{s+r}-v_r)}\,\Check M(v_1,\ldots, v_r)\,\Check N(v_{r+1}-v_r,\ldots, v_{s+r}-v_r) -\nline-
\tfrac{v_1}{v_{r+1}\,(v_1-v_{r+1})}\,\Check N(v_{r+1},\ldots, v_{r+s})\,\Check M(v_1-v_{r+1},\ldots,v_r-v_{r+1}) +\nline+
\tfrac{v_{r+s}}{v_s\,(v_{r+s}-v_s)}\,\Check N(v_1,\ldots, v_s)\,\Check M(v_{s+1}-v_s,\ldots v_{r+s}-v_s) +\nline+
\tfrac{v_r-v_{r+1}}{v_r\,v_{r+1}}\,\Check M(v_1,\ldots,v_r)\,\Check N(v_{r+1},\ldots,v_{r+s}) -\nline- 
\tfrac{v_s-v_{s+1}}{v_s\,v_{s+1}}\,\Check N(v_1,\ldots,v_s)\,\Check M(v_{s+1},\ldots,v_{r+s})
\end{align*}
paying attention~to the exchange of~$r$ and~$s$ when we switch~$M$ and~$N$. The discussion on page~\pageref{page:discussion} shows that it suffices~to check that the poles~$\frac{1}{v_{r}}$ and~$\frac{1}{v_s}$,~$\frac{1}{v_{r+1}}$ and~$\frac{1}{v_{s+1}}$ cancel out.
\par
Let us note that the alternality of~$\ari(M,N)$~together with the basic
fact that any alternal mould~$M$ satisfies~$M(u_1,\ldots,u_r)=(-1)^{r-1}
M(u_r,\ldots,u_1)$~\hbox{\cite[Lemma 2.5.3]{Sonline}} 
imply that it is enough~to deal with the case of~$\frac{1}{v_{r}}$ and~$\frac{1}{v_{s}}$, the other two being deduced from them~by applying the involution corresponding~to the symmetry around~$\tfrac{r+s+1}{2}$. And the pole~$\frac{1}{v_s}$ is deduced from~$\frac{1}{v_{r}}$ since the~total expression for~$\ari(M,N)$ is antisymmetric in~$M$ and~$N$, and thus in~$r$ and~$s$.

\medskip
The only case left~to check is that of the pole~$\frac{1}{v_{r}}$. The corresponding factor is
\begin{align*}
 & \tfrac{v_{s+r}}{v_{s+r}-v_r}\,\Check M(v_1,\ldots, v_r)\,\Check N(v_{r+1}-v_r,\ldots, v_{s+r}-v_r)+\nline+
\tfrac{v_r-v_{r+1}}{v_{r+1}}\,\Check M(v_1,\ldots,v_r)\,\Check N(v_{r+1},\ldots,v_{r+s})
\end{align*}
which is clearly zero when~$v_r=0$. This concludes the proof of Proposition~\ref{prop:a1}.
\end{proof}

We can now complete the proof of~Theorem~\ref{these}. We use the following
elementary result of mould theory~\cite[Lemma 2.4.1]{Sonline}, 
\begin{equation}\label{eq:A4}\ari\bigl(\swap(A),\swap(B)\bigr)=\swap\bigl(\ari(A,B)\bigr)\qquad
\forall\ A,B\in\ARI_{\push}\text.
% \eqno(A.4)\]
\end{equation}

Let~$A,B\in \ARI^{\sing}_{\underline{\al}/\underline{\al}}$.  
Then~$C=\ari(A,B)\in \ARI_{\underline{\al}/\underline{\al}}$~by~\eqref{th:maprops:5} of~Theorem~\ref{maprops}.  Consider the mould~$\swap(C)\in \swapbar{\ARI}$.
Again~by~\eqref{th:maprops:5} of~Theorem~\ref{maprops},~$\ARI_{\underline{\al}/\underline{\al}}$ is contained in~$\ARI_{\push}$, so the moulds
$A,B$ and~$C$ are all~$\push$-invariant. Thus~\eqref{eq:A4} holds, i.e.
$\swap(C)=\ari(\swap(A),\swap(B))$.  

Both~$\swap(A)$ and~$\swap(B)$ lie in~$\swapbar{\ARI}^{\sing}_{\al}$~by definition.  
Thus~by Proposition~\ref{prop:a1}, we also have~$\swap(C)\in
\swapbar{\ARI}^{\sing}_{\al}$ and thus~$C\in \ARI^{\sing}$.  But we also have
$C=\ari(A,B)\in \ARI_{\underline{\al}/\underline{\al}}$, so 
$C\in \ARI^{\sing}_{\underline{\al}/\underline{\al}}$, which concludes
the proof of~Theorem~\ref{these}.\qed

% \section*{Appendix~\ref{app:B}}
\section{Proof of~Theorem~\ref{B3}}\label{app:B}

This section is devoted~to the proof of~Theorem~\ref{B3}, which 
is stated as Proposition~\ref{prop:B2} below.

We use the notation and terminology of section~\ref{sec3}, with one further definition:
for any polynomial~$F\in\Lie[a,b]$ of homogeneous depth~$r$, 
we define~$\da(F)$~to be the mould given~by
\[\da(F)(u_1,\ldots,u_r)=\frac{\ma(F)}{u_1\cdots u_r}.\]
As we did for~$\Da(F)$, if~$F=\sum_r F^r$ is any~Lie polynomial broken up into 
its depth-graded parts, we define~$\da(F)=\sum_r \da(F^r)$.

Let~$\lu(A,B)$ be as in the beginning of~Appendix~\ref{app:A}.
For any polynomial~$U$, let~$\Darit_U$ be the operator 
on moulds defined~by
\[\Darit_U\cdot A=-\arit\bigl(\Da(U)\bigr)\cdot A-\lu\bigl(A,
\Da(U)\bigr).\]

\begin{prop}\label{prop:B1}
Let~$U$ be a push-invariant polynomial
in~$\Lie[a,b]$, and let~$D_U$ be the associated derivation.
Then for any~Lie polynomial~$F$ in~$a$ and~$b$, we have
\begin{equation}-\da\bigl(D_U(F)\bigr)=\Darit_U\cdot \da(F).\label{eq:B1}\end{equation}
\end{prop}

\begin{proof} By additivity, we may assume that~$U$ is of homogeneous
depth~$r$ and~$F$ is of homogeneous depth~$s$.  We will use induction on the
depth of~$F$~to
prove the proposition; the only difficult part is the base case~$s=1$.

\medskip
So assume that~$F$ is of depth~$1$, i.e.~$F=C_n=\ad(a)^{n-1}(b)$.
Here, since~$D_U([a,b])$ is zero, we have the explicit
formula
\[D_U(F)=\sum_{i=0}^{n-2} \ad(a)^i\ad(U)\ad(a)^{n-i-2}(b)
=\sum_{i=0}^{n-2} \ad(a)^i(UC_{n-i-1}-C_{n-i-1}U).\]
Applying~$\ma$~to both sides of this formula, and using the facts that
$\ma(C_{n-i-1})$ is equal to~$(-1)^{n-i-2}u_1^{n-i-2}$ and that
for a~Lie polynomial~$P$ of homogeneous depth~$r$ we have
\[\ma([a,P])=-(u_1+\cdots+u_r)\ma(P)\]
by~\eqref{th:maprops:3} of~Theorem~\ref{maprops},
and using the identity
\begin{align*}\ma(UC_{n-i-1}-C_{n-i-1}U)&=
(-1)^{n-i-2}\ma(U)(u_1,\ldots,u_r)u_{r+1}^{n-i-2}\\
&\ \ \ \ -(-1)^{n-i-2}u_1^{n-i-2}\ma(U)(u_2,\ldots,u_{r+1}),\end{align*}
we find that
\begin{align*}\ma\bigl(D_U(F)\bigr)&=
\sum_{i=0}^{n-3} (-1)^{n-1}(u_1+\cdots+u_{r+1})^i \cdot\\
&\ \ \ \Bigl( \ma(U)(u_1,\ldots,u_r)u_{r+1}^{n-i-2}-u_1^{n-i-2}\ma(U)(u_2,\ldots,u_{r+1})\Bigr),
\end{align*}
so, with a little rewriting of indices, we find that the left-hand side of~\eqref{eq:B1} equals
\begin{align*}-\da\bigl(D_U(F)\bigr)&=(-1)^{n-1}{\frac{\ma(U)(u_2,\ldots,u_{r+1})}{\De_{r+1}}}\sum_{i=1}^{n-2} (u_1+\cdots+u_{r+1})^iu_1^{n-i-1}\\
&\ \ \ \ -(-1)^{n-1}{\frac{\ma(U)(u_1,\ldots,u_r)}{\De_{r+1}}}\sum_{i=1}^{n-2}
(u_1+\cdots+u_{r+1})^iu_{r+1}^{n-i-1},
\end{align*}
or, adding up the sums~to obtain a closed expression,
\begin{gather}-\da\bigl(D_U(F)\bigr)=(-1)^{n-1}{\frac{\ma(U)(u_2,\ldots,u_{r+1})}{\De_{r+1}}}
\Bigl({\tfrac{(u_1+\cdots+u_{r+1})^{n-1}u_1-u_1^{n-1}(u_1+\cdots+u_{r+1})}{(u_2+\cdots+u_{r+1})}}\Bigr)\nonumber\\
\ \ \ \ \ -(-1)^{n-1}{\frac{\ma(U)(u_1,\ldots,u_r)}{\De_{r+1}}}
\Bigl({\tfrac{(u_1+\cdots+u_{r+1})^{n-1}u_{r+1}-u_{r+1}^{n-1}(u_1+\cdots+u_{r+1})}{(u_1+\cdots+u_r)}}
\Bigr) \label{eq:B2}\end{gather}

Let us now compute the right-hand side of~\eqref{eq:B1}. We
have~$\ma(F)=(-1)^nu_1^{n-1}$, so~$\da(F)=(-1)^nu_1^{n-2}$; thus
\begin{equation}\label{eq:B3}\lu\bigl(\da(F),\Da(U)\bigr)=
(-1)^n\bigl({\tfrac{u_1^{n-2}\ma(U)(u_2,\ldots,u_{r+1})}{u_2\cdots u_{r+1}(u_2+\cdots+
u_{r+1})}}\bigr) -(-1)^n \bigl(\tfrac{\ma(U)(u_1,\ldots,u_r)u_{r+1}^{n-2}}{u_1\cdots u_r(u_1+
\cdots+u_r)}\bigr)\text.
\end{equation}
Also, we have
\begin{align}\arit&\bigl(\Da(U)\bigr)\cdot \da(F)=
\!\!\!\!\!\sum_{w=abc, c\ne\varnothing} \da(F)(a\lceil c)\Da(U)(b)-
\!\!\!\!\!\sum_{w=abc, a\ne\varnothing} \da(F)(a\rceil c)\Da(U)(b)\nonumber\\
&=(-1)^n(u_1+\cdots+u_{r+1})^{n-2}\Da(U)(u_1,\ldots,u_r)-{}\nonumber\\&\qquad{}
-(-1)^n(u_1+\cdots+u_{r+1})^{n-2}\Da(U)(u_2,\ldots,u_{r+1})\nonumber\\
&=(-1)^n(u_1+\cdots+u_{r+1})^{n-2}\Bigl({\tfrac{\ma(U)(u_1,\ldots,u_r)}{u_1\cdots u_r
(u_1+\cdots+u_r)}}
-{\tfrac{\ma(U)(u_2,\ldots,u_{r+1})}{u_2\cdots u_{r+1}
(u_2+\cdots+u_{r+1})}}\Bigr)\text,\label{eq:B4}
\end{align}
since~$\da(F)$ is a depth~$1$ mould and therefore the only possible
decomposition~$w=abc$ with~$c\ne\varnothing$ is~$a=\varnothing$,~$b=(u_1,\ldots,u_r)$
and~$c=u_{r+1}$, and the only possible decomposition with~$a\ne\varnothing$
is~$a=u_1$,~$b=(u_2,\ldots,u_r)$ and~$c=\varnothing$.
We add~\eqref{eq:B3} and~\eqref{eq:B4}~to get the right-hand side of~\eqref{eq:B1}, obtaining
\begin{gather}\Darit_U\cdot\da(F)=(-1)^n{\frac{\ma(U)(u_1,\ldots,u_r)}{u_1\cdots u_r (u_1+\cdots+u_r)}}
\Bigl((u_1+\cdots+u_{r+1})^{n-2}-u_{r+1}^{n-2}\Bigr)
\nonumber\\
\ \ \ \ \ +(-1)^n{\frac{\ma(U)(u_2,\ldots,u_{r+1})}{u_2\cdots u_{r+1} (u_2+\cdots+u_{r+1})}}
\Bigl(u_1^{n-2}-(u_1+\cdots+u_{r+1})^{n-2}\Bigr)
\nonumber\\
=(-1)^n{\frac{\ma(U)(u_1,\ldots,u_r)}{\De_{r+1}}}
\Bigl({{(u_1+\cdots+u_{r+1})^{n-1}u_{r+1}-u_{r+1}^{n-1}}\tfrac{(u_1+\cdots+u_{r+1})}{(u_1+\cdots+u_r)}}\Bigr)
\nonumber\\
\qquad -(-1)^n{\frac{\ma(U)(u_2,\ldots,u_{r+1})}{\De_{r+1}}}
\Bigl({\tfrac{u_1(u_1+\cdots+u_{r+1})^{n-1}-u_1^{n-1}(u_1+\cdots+u_{r+1})}
{(u_2+\cdots+u_{r+1})}}\Bigr)\text,\label{eq:B5}
\end{gather}
which is equal~to~\eqref{eq:B2}.
This settles the base case where~$F$ is of depth~$1$.

Now assume 
that~\eqref{eq:B1} holds up~to depth~$s-1$, and let~$F$ be a~Lie polynomial of
depth~$s$.  Then~$F$ is a linear combination of~Lie brackets, so~by additivity,
we may assume that~$F$ is a single~Lie bracket; thus~$F=[G,H]$ for~Lie
brackets~$G,H$ that are of depth~$<s$.
We saw in~Appendix~\ref{app:A} that~$\ma(FG)=\muu(\ma(F),\ma(G))$,
so~by definition of~$\da$ we also have~$\muu(\da(F),\da(G))=\da(FG)$.
Furthermore
$\Darit_U$ is a derivation for the~$\lu$-bracket since both
$\arit(B)$ and~$\lu(\text{\vartextvisiblespace[.75em]},B)$ are~--- this is obvious for~$\lu$ but
difficult for~$\arit$, cf. \cite[Prop.~2.2.1]{Sonline}~---,
so using the induction hypothesis for~$G$ and~$H$, we have 
\begin{align*}-\da\bigl(D_U(F)\bigr)&=
-\da([D_U(G),H]+[G,D_U(H)])\\
&=-\lu\bigl(\da(D_U(G)),\da(H)\bigr)-\lu\bigl(\da(G),\da(D_U(H))\bigr)\\
&=\lu\bigl(\Darit_U(\da(G))),\da(H)\bigr)+\lu\bigl(\da(G),\Darit_U(\da(H)))\bigr)\\
&=\Darit_U\cdot \lu\bigl(\da(G)),\da(H)\bigr)\\
&=\Darit_U\cdot \da([G,H])\\
&=\Darit_U\cdot \da(F).
\end{align*}
This proves the proposition.
\end{proof}

\begin{prop}\label{prop:B2}
% \noindent {\bf Proposition B.2.} {\it 
The map
$\Der^0\Lie[a,b]\rightarrow \ARI$ given~by
$D_U\mapsto \Da(U)$ is a~Lie algebra morphism, i.e.
\begin{equation}\label{eq:B6}
[D_U,D_V]\mapsto \ari\bigl(\Da(U),\Da(V)\bigr)\text.
\end{equation}
\end{prop}

\begin{proof}
It is easily seen that
\begin{equation}\label{eq:B7}
\arit(B)\circ\arit(A)-\arit(A)\circ \arit(B)=\arit\bigl(\ari(A,B)\bigr)\text.
\end{equation}
Using only this identity, the Jacobi relation and~\eqref{eq:B1}, we compute
\begin{align*}-&\da\bigl([D_U,D_V](F)\bigr)
=-\da\bigl(D_U(D_V(F))\bigr)
+\da\bigl(\ma(D_V(D_U(F))\bigr)\\
&=\arit\bigl(\Da(U)\bigr)\cdot \da(D_V(F))
+\lu\bigl(\da(D_V(F)),\Da(U)\bigr)\\
&\ \ \ \ +\arit\bigl(\Da(V)\bigr)\cdot \da(D_U(F))
+\lu\bigl(\da(D_U(F)),\Da(V)\bigr)\\
&=-\arit\bigl(\Da(U)\bigr)\cdot \Bigl(\arit(\Da(V))\cdot
\da(F)+\lu\bigl(\da(F),\Da(V)\bigr)\Bigr)\\
&\ \ \ \ \ -\lu\bigl(\arit(\Da(V))\cdot \da(F),\Da(U)\bigr)
-\lu\bigl(\lu(\da(F),\Da(V)),\Da(U)\bigr)\\
&\ \ \ \ \ +\arit\bigl(\Da(V)\bigr)\cdot \Bigl(\arit(\Da(U))\cdot
\da(F)+\lu\bigl(\da(F),\Da(U)\bigr)\Bigr)\\
&\ \ \ \ \ +\lu\bigl(\arit(\Da(U))\cdot \da(F),\Da(V)\bigr)
-\lu\bigl(\lu(\da(F),\Da(U)),\Da(V)\bigr)\\
&=\arit\bigl(\ari(\Da(U),\Da(V))\bigr)\cdot \da(F)\qquad\qquad(\textnormal{by }\eqref{eq:B7})\\
&\ \ \ -\lu\Bigl(\da(F),\arit\bigl(\Da(U)\bigr)\cdot \Da(V)\Bigr)
+\lu\Bigl(\da(F),\arit(\Da(V))\cdot \Da(U)\Bigr)\\
&\ \ \ -\lu\Bigl(\lu(\da(F),\Da(V)),\Da(U)\Bigr)
+\lu\Bigl(\lu\bigl(\da(F),\Da(U)\bigr),\Da(V)\Bigr)\\
&=\arit\Bigl(\ari\bigl(\Da(U),\Da(V)\bigr)\Bigr)\cdot \da(F)\\
&\ \ \ \ \ -\lu\Bigr(\da(F),\arit\bigl(\Da(U)\bigr)\cdot \Da(V)\Bigl)
+\lu\Bigl(\da(F),\arit\bigl(\Da(V)\bigr)\cdot \Da(U)\Bigr)\\
&\ \ \ \ \ +\lu\Bigl(\da(F),\lu\bigl(\Da(U),\Da(V)\bigr)\Bigr)
\qquad\qquad{\rm (by\ Jacobi)}\\
&=\arit\Bigl(\ari\bigl(\Da(U),\Da(V)\bigr)\Bigr)\cdot \da(F)\\
&\ \ \ +\lu\Bigr(\da(F),-\arit\bigl(\Da(U)\bigr)\cdot \Da(V)
+\arit(\Da(V))\cdot \Da(U)
+\lu\bigl(\Da(U),\Da(V)\bigr)\Bigr)\\
&=\arit\Bigl(\ari\bigl(\Da(U),\Da(V)\bigr)\Bigr)\cdot \da(F)
+\lu\Bigr(\da(F),\ari\bigl(\Da(U),\Da(V)\bigr)\Bigr)\\
&=\Darit_W\cdot\da(F),
\end{align*}
where~$W$ is the polynomial such that~$\Da(W)=\ari\bigl(\Da(U),\Da(V)\bigr)$
(obtained~by multiplying the right-hand mould~by~$\De$ and inverting
the injective map~$\ma$).  However, the above calculation shows that 
\[\Darit_W\cdot\da(F)= -\da\bigl([D_U,D_V](F)\bigr)=\Darit_P\cdot\da(F),\]
where~$P=D_U(V)-D_V(U)$, that is~$D_P=[D_U,D_V]$. Thus
the two derivations~$\Darit_W$ and~$\Darit_P$ take the same values on
all elements~$F$ of~$\Lie[a,b]$, so~$W=P$, i.e.~$[D_U,D_V]=D_P=D_W$,
so under the map~$D_U\mapsto \Da(U)$ of~\eqref{eq:B6}, we have~$[D_U,D_V]=D_W\mapsto \Da(W)=\ari\bigl(\Da(U),\Da(V)\bigr)$. This proves~Proposition~\ref{prop:B2}.
\end{proof}

\end{document}